\documentclass[12]{amsart}

\usepackage{geometry}
\usepackage{amsmath,amssymb,}
\newtheorem{thm}{Theorem}[section]

\newtheorem{cor}[thm]{Corollary}
\newtheorem{lem}[thm]{Lemma}
\newtheorem{ques}[thm]{Questions}

\newtheorem{prop}[thm]{Proposition}
\newtheorem{exa}[thm]{Example}

\theoremstyle{definition}
\newtheorem{defn}[thm]{Definition}

\numberwithin{equation}{section}

\begin{document}
	
	\title[Rings whose units are nil-clean]{Rings in which every unit is a sum of a nilpotent and an idempotent}

	\author[Karimi-Mansoub, Kosan, Zhou]{Arezou Karimi-Mansoub, Tamer  Ko\c{s}an and Yiqiang Zhou}

	\address{Department of Pure Mathematics, Faculty of Mathematical Sciences, Tarbiat
Modares University, Tehran, Iran, P.O. Box 14115-134}
\email{arezoukarimimansoub@gmail.com}

\address{Department of Mathematics, Gebze Technical University,  Gebze/Kocaeli, Turkey}
\email{mtkosan@gtu.edu.tr}
	\address{Department of Mathematics and Statistics, Memorial University of Newfoundland, St.John's, Nfld A1C 5S7, Canada} 
\email{zhou@mun.ca}
	

\baselineskip=20pt
\maketitle
\centerline{\it Dedicated to the memory of Bruno J. M\"{u}ller}

\begin{abstract} A ring $R$ is a UU ring if every unit is unipotent, or equivalently if every unit is a sum of a nilpotent and an idempotent that commute.
These rings have been investigated in  C\u{a}lug\u{a}reanu \cite{C} and in
Danchev and Lam \cite{DL}. In this paper, two generalizations of UU rings are discussed. We study rings for which every unit is a sum of a nilpotent and an idempotent, and rings for which every unit is a sum of a nilpotent and two idempotents that commute with one another.
	\end{abstract}

	\bigskip
	\section{Introduction}     
The motivation of this paper is a recent work of  Danchev and Lam \cite{DL} on UU rings. 	 	
Throughout, $R$ is an associative ring with identity. We denote by $J(R), U(R)$, ${\rm Nil}(R)$ and ${\rm idem}(R)$ the Jacobson radical, the unit group, the set of nilpotent elements and the set of idempotents of $R$, respectively. In \cite{D13}, Diesl introduced (strongly) nil-clean elements and rings as follows. An element $a$ in a ring $R$ is called (strongly) nil-clean if $a$ is the sum of an idempotent and a nilpotent (that commute with each other), and the ring is called (strongly) nil-clean if each of its elements is (strongly) nil-clean. 
One of the results in \cite{D13} states that a ring $R$ is strongly nil-clean if and only if $R$ is strongly $\pi$-regular with $U(R)=1+{\rm Nil}(R)$. This  motivated C\u{a}lug\u{a}reanu \cite{C} to introduce and study UU rings (rings
whose units are unipotent). Equivalently, a ring is a UU ring if and only if every unit is strongly nil-clean.  These rings have been extensively investigated 
in  Danchev and Lam \cite{DL}, where, among others, it is proved that a ring is strongly nil-clean if and only if it is an exchange (clean) UU ring, and it is asked 
whether a clean ring $R$ is nil-clean if and only if  every unit of $R$ is nil-clean. Here we are motivated to  study rings whose units are nil-clean. These rings will be called UNC rings. 

In section 2, we first prove several basic properties of UNC rings. Especially it is proved that every semilocal UNC ring is nil-clean. This can be seen as a partial answer to the question of Danchev and Lam.  We next show that the matrix ring over a commutative ring $R$ is a UNC ring if and only if $R/J(R)$ is Boolean with $J(R)$ nil. As a consequence, the matrix ring over a UNC ring need not be a UNC ring. We also discuss when a group ring is a UNC (UU) ring. In the last part of this section, it is shown that
UU rings are exactly those rings whose units are uniquely nil-clean. As the main result in this section, it is proved that a ring $R$ is strongly nil-clean if and only if  
$R$ is a semipotent UNC ring.

As another natural generalization of UU rings, in section 3 we determine the rings for which every unit is a sum of a nilpotent and two idempotents that commute with one another. We also deal with a special case where every unit of the ring is a sum of two commuting idempotents. These conditions can be compared with the so-called strongly $2$-nil-clean rings introduced by Chen and Sheibani \cite{CS}, and the rings for which every element is a sum of two commuting idempotents, studied by  Hirano and  Tominaga in \cite{HT}.

We write ${\mathbb M}_n(R)$, ${\mathbb T}_n(R)$ and $R[t]$ for the $n\times n$ matrix ring, the $n\times n$ upper triangular matrix ring, and the polynomial ring over $R$, respectively. For an endomorphism $\sigma$ of a ring $R$, let $R[t;\sigma]$
denote the ring of left skew power series over $R$. Thus, elements of $R[t;\sigma]$ are polynomials in $t$
with coefficients in $R$ written on the left, subject to the relation $tr = \sigma(r)t$ for all $r\in R$. 
The group ring of a group $G$ over a ring $R$ is denoted by $RG$.

\section{Units being nil-clean}
\subsection{Basic properties} 
We present various properties of the rings whose units are nil-clean, and prove that
every semilocal ring whose units are nil-clean is a nil-clean ring.
\begin{defn} \label{defn2.1}
A ring $R$ is called a UNC ring if every unit of $R$ is nil-clean.
\end{defn}

Every nil-clean ring is a UNC ring. A ring $R$ is called a UU ring if $U(R)=1+{\rm Nil}(R)$ (see \cite{C}).

\begin{prop}\cite{D13}\label{prop2.2} A unit $u$ of $R$ is strongly nil-clean if and only if $u\in 1+{\rm Nil}(R)$. 
In particular, $R$ is a UU ring if and only if every unit of $R$ is strongly nil-clean.
\end{prop}

Thus, UU rings can be viewed as the ``strong version'' of UNC rings. 
\begin{lem}\label{lem2.3} The class of UNC rings is closed under finite direct sums.
\end{lem}
 The next lemma was proved for a nil-clean ring in \cite{D13} and for a UU ring in \cite{DL}. 

\begin{lem}\label{lem2.4} If $R$ is a UNC ring, then $J(R)$ is nil and $2\in J(R)$.
\end{lem}
\begin{proof}
Let $j\in J(R)$. Then $1+j=e+b$ where $e^2=e$ and $b\in {\rm Nil}(R)$, and so $e=(1-b)+j\in U(R)$. It follows that $e=1$, and hence $j=b\in {\rm Nil}(R)$. Write $-1=e+b$ where $e^2=e$ and $b\in {\rm Nil}(R)$. Then $e=-1-b\in U(R)$. So $e=1$ and hence $2=-b\in {\rm Nil}(R)$. 
\end{proof}

The next result is basic for studying the structure of a UNC ring. 
\begin{thm}\label{thm2.5} Let $R$ be a ring, and $I$ a nil ideal of $R$. 
\begin{enumerate}
\item $R$ is a UNC ring if and only if $J(R)$ is nil and $R/J(R)$ is a UNC ring.
\item $R$ is a UNC ring if and only if $R/I$ is a UNC ring.
\end{enumerate}
\end{thm}
\begin{proof}
$(1)$ The necessity is clear in view of Lemma \ref{lem2.4}. For the sufficiency, let $u\in U(R)$. Then $\bar u\in U(R/J(R))$, and write $\bar u=\bar e+\bar b$ where $\bar e\in {\rm idem}(R/J(R))$ and $\bar b\in {\rm Nil}(R/J(R))$. As $J(R)$ is nil, idempotents of $R/J(R)$ can be lifted to idempotents of $R$. So we can assume that $e^2=e\in R$. Moreover, $b\in R$ is nilpotent. 
Thus, for some $j\in J(R)$, $u=e+b+j=e+(b+j)$ is nil-clean because $b+j\in {\rm Nil}(R)$. 

$(2)$ The proof is similar to $(1)$.
\end{proof}

The following corollary can be quickly verified using Theorem \ref{thm2.5}.
\begin{cor}\label{cor2.6}
Let $R, S$ be rings, $M$ be an $(R,S)$-bimodule, and $N$ a bimodule over $R$.
\begin{enumerate}
\item The trivial extension $R\propto N$ is a UNC ring if and only if $R$ is a UNC ring.
\item For $n\ge 2$, $R[t]/(t^n)$  is UNC ring if and only if $R$ is a UNC ring.
\item The formal triangular matrix ring $\begin{pmatrix}R&M\\
                       0&S\end{pmatrix}$ is a UNC ring if and only if $R, S$ are UNC rings. 
\item For $n\ge 1$, ${\mathbb T}_n(R)$ is a UNC ring if and only if $R$ is a UNC ring. 
\end{enumerate}
\end{cor}

The easiest way to see a UNC ring that is neither UU nor nil-clean is to form the ring direct sum $R=R_1\oplus R_2$, 
where $R_1$ is a UU ring that is not nil-clean and $R_2$ is a nil-clean ring that is not UU. For instance,  the ring $\mathbb Z_2[t]\oplus \mathbb M_2(\mathbb Z_2)$ is a UNC ring that is neither UU nor 
nil-clean. The next example gives an indecomposable UNC ring that is neither UU nor nil-clean.
\begin{exa}\label{exa2.7}
Let $R={\mathbb M}_n(\mathbb Z_2)$ with $n\ge 2$, $S=\mathbb Z_2[t]$ and $M=S^n$. For $(a_{ij})\in R$, $b\in S$ and 
$\vec x=\left[ 
\begin{array}{c}
x_1\\ 
x_2\\
\vdots\\
x_n\\
\end{array}
\right] 
\normalsize\in S^n$, define $(a_{ij})\vec x=\left[ 
\begin{array}{c}
a_{11}x_1+a_{12}x_2+\cdots+a_{1n}x_n\\ 
a_{21}x_1+a_{22}x_2+\cdots+a_{2n}x_n\\ 
\vdots\\
a_{n1}x_1+a_{n2}x_2+\cdots+a_{nn}x_n\\ 
\end{array}
\right] 
\normalsize$ and $\vec x b=\left[ 
\begin{array}{c}
x_1b\\ 
x_2b\\
\vdots\\
x_nb\\
\end{array}
\right] 
\normalsize$. Then $M$ is an $(R,S)$-bimodule, and the formal triangular matrix ring $T:=\begin{pmatrix}R&M\\
                       0&S\end{pmatrix}$ is a UNC ring by Corollary \ref{cor2.6}(3). One can verify that
the central idempotents of $T$ are trivial. So $T$ is indecomposable. 											
But since $R$ is not a UU ring and $S$ is not a nil-clean ring, $T$ is neither a UU ring nor a nil-clean ring.
\end{exa}

\begin{lem}\cite{STZ, Ster}\label{lem2.8}
An element $a\in R$ is strongly nil-clean if and only if $a-a^2$ is nilpotent.
\end{lem}

\begin{cor}\label{cor2.9} Let $C(R)$ be the center of $R$.
\begin{enumerate}
\item If $R$ is nil-clean, then $C(R)$ is strongly nil-clean.
\item If $R$ is a UNC ring, then $C(R)$ is a UU ring.
\end{enumerate}
\end{cor}
\begin{proof}
$(1)$ Let $a\in C(R)$. Then $a\in R$ is nil-clean and central, so $a$ is strongly nil-clean in $R$. Thus $a^2-a$ is nilpotent by Lemma \ref{lem2.8}. Hence $C(R)$ is strongly nil-clean by Lemma \ref{lem2.8}.

$(2)$ The proof is similar to $(1)$.
\end{proof}

The proof of the next corollary actually shows that, if $2\in J(R)$, then $a\in R$ is strongly nil-clean if and only if $a$ can be written as $a=e+b$, where $e^2=e\in R$, $b\in {\rm Nil}(R)$ and $eb=ebe$.

\begin{cor} \label{cor2.10}The following hold for a ring $R$:
\begin{enumerate}
\item  $R$ is strongly nil-clean if and only if, for each $a\in R$, $a$ can be written as $a=e+b$, where $e^2=e\in R$, $b\in {\rm Nil}(R)$ and $eb=ebe$.
\item $R$ is a UU ring if and only if, for each $a\in U(R)$, $a$ can be written as $a=e+b$, where $e^2=e\in R$, $b\in {\rm Nil}(R)$ and $eb=ebe$.
\end{enumerate}
\end{cor}
\begin{proof} $(1)$ We show the sufficiency. For $a\in R$,  let $a=e+b$ as given as in $(1)$. Then $a^2=e+eb+be+b^2$, so $a-a^2=d-x$ where $d=b-b^2$ and $x=eb+be$. We have $ex=eb+ebe=2eb\in J(R)$ as $2\in J(R)$ (see \cite{D13}). It also follows from $eb=ebe$ that $eb^n=eb^ne$ for any $n\ge 0$, so $eb^nx=eb^nex\in J(R)$. Hence, we deduce that $xd^nx\in J(R)$ for all $n\ge 0$. We now show that $a-a^2$ is a nilpotent. As $J(R)$ is nil by \cite{D13}, it suffices to show that $\bar a-\bar a^2$ is nilpotent in $R/J(R)$. As $\bar x{\bar d}^n\bar x=\bar 0$ in $R/J(R)$ for all $n\ge 0$, we have $(\bar a-\bar a^2)^{m+1}=(\bar d-\bar x)^{m+1}=\bar d^{m+1}-\sum_{i+j=m}\bar d^i\bar x\bar d^j$ for any $m>0$.
So if $\bar d^m=\bar 0$, then $(\bar a-\bar a^2)^{2m+1}=\bar 0$. 

$(2)$ The proof is similar to $(1)$.
\end{proof}

By \cite{C}, a commutative ring $R$ is a UU ring if and only if so is $R[t]$. Next we present a generalization of this result. The prime radical $N(R)$  of a ring $R$ is defined to be  the intersection of the prime ideals of $R$. It is known that $N(R)={\rm Nil}_{*}(R)$, the lower nilradical of $R$. A ring $R$ is called a $2$-primal ring if $N(R)$ coincides with ${\rm Nil}(R)$. For an endomorphism $\sigma$ of $R$, $R$ is called $\sigma$-compatible if, for any $a,b\in R$, $ab=0\, \Leftrightarrow \, a\sigma(b)=0$ (see \cite{A}), and in  this case $\sigma$ is clearly injective.

\begin{thm} \label{thm2.11} Let $R$ be a $2$-primal ring and $\sigma$ an endomorphism of $R$ such that $R$ is $\sigma$-compatible. The following are equivalent:
\begin{enumerate}
\item $R[t;\sigma]$ is a UNC ring.
\item $R[t;\sigma]$ is a UU ring.
\item $R$ is a UU ring.
\item $R$ is a UNC ring.
\item $J(R)={\rm Nil}(R)$ and $U(R)=1+J(R)$.
\end{enumerate}
\end{thm}
\begin{proof}
$(2)\Rightarrow (1)$ and $(3)\Rightarrow (4)$. The implications are obvious.

$(1)\Rightarrow (4)$. As $R[t;\sigma]/(t)\cong R$ and units of $R[t;\sigma]/(t)$ are lifted to units of $R[t;\sigma]$, the implication holds.

$(2)\Rightarrow (3)$. Argue as in proving $(1)\Rightarrow (4)$.

$(4)\Rightarrow (5)$. As $R$ is $2$-primal, ${\rm Nil}(R)\subseteq J(R)$, so $J(R)={\rm Nil}(R)$ by Lemma \ref{lem2.4}.  Hence $R/J(R)$ is a reduced ring
that is a UNC ring. It follows that every unit of $R/J(R)$ is an idempotent, so $U(R/J(R))=\{\bar 1\}$. That is, 
$U(R)=1+J(R)$.

$(5)\Rightarrow (2)$. 
As $R$ is a $2$-primal ring, we deduce from $(5)$ that $J(R)={\rm Nil}_{*}(R)={\rm Nil}(R)$. So $R/J(R)$ is a reduced ring. As $\sigma(J(R))\subseteq J(R)$, 
$\overline{\sigma}: R/J(R)\rightarrow R/J(R)$ defined by $\overline {\sigma}(\bar a)=\overline {\sigma(a)}$ is an endomorphism of $R/J(R)$. We next show that $R/J(R)$ is 
$\overline{\sigma}$-compatible. This is, for any $a,b\in R$, $ab\in {\rm Nil}(R) \Leftrightarrow a\sigma(b)\in {\rm Nil}(R)$.  But this equivalence ``$\Leftrightarrow$'' has been established in the proof of Claims 1 and 2 of \cite[Theorem 3.6]{AMH} (also see \cite[Lemma 2.1]{Ch}). As $R/J(R)$ is a reduced ring that is $\overline \sigma$-compatible, by \cite[Corollary 2.12]{Ch} we have $U\big(\frac{R}{J(R)}[t;\overline{\sigma}]\big)=U\big(\frac{R}{J(R)}\big)$, which is equal to $\{\bar 1\}$ by $(5)$.  As $\frac{R[t;\sigma]}{J(R)[t;\sigma]}\cong \frac{R}{J(R)}[t;\overline{\sigma}]$, it follows that $\frac{R[t;\sigma]}{J(R)[t;\sigma]}$ is a UU ring. As $R$ is a $2$-primal ring that is $\sigma$-compatible, ${\rm Nil}_{*}(R[t;\sigma])={\rm Nil}_{*}(R)[t;\sigma]$ by \cite[Lemma 2.2]{Ch}. We infer that $J(R)[t;\sigma]={\rm Nil}_{*}(R[t;\sigma])$, which is nil. Hence, by \cite[Theorem 2.4]{DL} that $R[t;\sigma]$ is a UU ring.  
\end{proof}

\begin{cor} \label{cor2.12} A $2$-primal ring $R$ is a UNC ring if and only if $R[t]$ is a UNC ring, if and only if $J(R)={\rm Nil}(R)$ and $U(R)=1+J(R)$.
\end{cor}

By \cite{DL}, a ring $R$ is strongly nil-clean if and only if $R$ is a clean, UU ring. It is asked in \cite{DL} whether a clean, UNC ring is nil-clean (the converse holds clearly). We show that every semilocal UNC ring is nil-clean. In particular, a semiperfect, UNC ring is nil-clean.
The following lemma is implicit in the proof of \cite[Theorem 3]{KLZ14}.

\begin{lem} \cite{KLZ14}\label{lem2.13}
Let $D$ be a division ring.  If $|D|\ge 3$ and $a\in D\backslash \{0, 1\}$, then $\begin{pmatrix}a&\bf 0\\
                       \bf 0&\bf 0\end{pmatrix}\in {\mathbb M}_n(D)$ is not 
nil-clean.
\end{lem}

The equivalence $(2)\Leftrightarrow (3)$ of the next theorem was proved in \cite{KLZ14}.

\begin{thm}\label{thm2.14}
The following are equivalent for a semilocal ring $R$:
\begin{enumerate}
\item $R$ is a UNC ring.
\item $R$ is a nil-clean ring.
\item $J(R)$ is nil and $R/J(R)$ is a finite direct sum of matrix rings over $\mathbb F_2$.
\end{enumerate}
\end{thm}

\begin{proof}
$(2)\Leftrightarrow (3)$. The equivalence is \cite[Corollary 5]{KLZ14}. The implication $(2)\Rightarrow (1)$ is clear.

$(1)\Rightarrow (3)$. By Lemma \ref{lem2.4}, $J(R)$ is nil. As $R$ is semilocal, $R/J(R)=R_1\oplus R_2\oplus \cdots\oplus R_n$ where each $R_i$ is a matrix ring over a division ring $D_i$. If $|D_i|>2$, let $a\in D_i\backslash \{0,1\}$. Then, for any $n\ge 1$,  $A:=\begin{pmatrix}a&\bf 0\\
                       \bf 0&\bf 0\end{pmatrix}\in {\mathbb M}_n(D_i)$ is not nil-clean by Lemma \ref{lem2.13}, and 
											this implies that $I_n-A$ is not nil-clean in ${\mathbb M}_n(D_i)$. But $I_n-A$ is a unit of 
 ${\mathbb M}_n(D_i)$, we deduce that ${\mathbb M}_n(D_i)$ is not a UNC ring. Hence $R_i$ is not a UNC ring, so $R$ is not a UNC ring by Lemma \ref{lem2.3} and Theorem \ref{thm2.5}. This contradiction shows that $|D_i|=2$.										
\end{proof}

\subsection{UNC matrix rings}
Any proper matrix ring can not be a strongly nil-clean ring by \cite{D13} (indeed, can not be a UU ring by \cite{C}). It is still unknown whether 
the matrix ring over a nil-clean ring is again nil-clean (see \cite[Question 3]{D13}).
Next we determine when the matrix ring over a commutative ring is a UNC ring. As a consequence, the matrix ring over a UNC ring need not be a UNC ring.

The $(i,j)$-cofactor of an $n\times n$ matrix $A$ over a commutative ring, denoted by $A_{ij}$, is $(-1)^{i+j}$ times the determinant of the $(n-1)\times (n-1)$ submatrix obtained from $A$ by deleting row $i$ and column $j$. Let $E_{ij}$ be the square matrix with $(i,j)$-entry $1$ and all other entries $0$.
  
\begin{lem}\cite{KSZ}\label{lem3.1}
Let $R$ be a commutative ring and let $n\ge 1$. If $A\in {\mathbb M}_n(R)$ and $x\in R$, then ${\rm det}(xE_{ij}+A)=xA_{ij}+{\rm det}(A)$. 
\end{lem}

\begin{thm} \label{thm3.2}Let $R$ be a commutative ring and $n\ge 2$. Then 
${\mathbb M}_n(R)$ is a UNC ring if and only if $J(R)$ is nil and $R/J(R)$ is Boolean. 
\end{thm}
\begin{proof} $(\Leftarrow)$. This is by \cite[Corollary 6.2]{KWZ16}. 

$(\Rightarrow )$. By Lemma \ref{lem2.4},  ${\mathbb M}_n(J(R))$ is nil, so $J(R)$ is nil. Thus, it suffices to show that $R/J(R)$ is Boolean.
As ${\mathbb M}_n(R/J(R))\cong {\mathbb M}_n(R)/J({\mathbb M}_n(R))$ is a UNC ring (by Theorem \ref{thm2.5}) and $R/J(R)$ is reduced, we can assume without loss of generality that $R$ is reduced. We next show that $R$ is Boolean. 
As every commutative reduced ring is a subdirect product of integral domains, there exist a family of ideals $\{I_\alpha\}$ of $R$ such that $\cap I_\alpha=0$ and $R/I_\alpha$ is an integral domain for each $\alpha$. To show that $R$ is Boolean, it suffices to show that each $R/I_\alpha\cong \mathbb Z_2$. 

Firstly, we show that $R/I_\alpha$ is a field. For $a\in R$, write $\overline a=a+I_\alpha\in R/I_\alpha$. For $A=(a_{_{ij}})\in {\mathbb M}_n(R)$, write $\overline A=(\overline a_{_{ij}})\in {\mathbb M}_n(R/J(R))$. Assume that $R/I_\alpha$ is not a field. Then there exists $\bar 0\not= \overline x\in R/I_\alpha$ such that $\overline x\notin U(R/I_\alpha)$. The matrix
 \begin{equation*}
\begin{split}
U:&=I_n+xE_{12}+xE_{21}+x^2E_{22}\\
&=\begin{pmatrix}1&x&0&\cdots &0\\
                       x&1+x^2&0&\cdots &0\\
                       0&0&1&\cdots &0\\
                       \vdots&\vdots&\vdots&\ddots&\vdots\\
                       0&0&0&\cdots&1\end{pmatrix}
\end{split}
\end{equation*} is a unit in $S:={\mathbb M}_n(R)$, so $U$ is nil-clean in $S$. Hence $\overline U$ is nil-clean in $\overline S={\mathbb M}_n(R/I_\alpha)$. Write  $\overline U=\epsilon+\beta$ where $\epsilon^2=\epsilon\in \overline S$ and $\beta\in {\rm Nil}(\overline S)$.  One easily sees that $\epsilon\not= I_n$. So ${\rm det}(\epsilon)\not= \bar 1$, and hence ${\rm det}(\epsilon)=\bar 0$. 
Thus, $\bar 0={\rm det}(\epsilon)={\rm det}(\overline U-\beta)={\rm det}\big(\overline xE_{12}+\overline xE_{21}+{\overline x}^2E_{22}+(I_n-\beta)\big)$. By Lemma \ref{lem3.1}, there exist $a,b,c\in R/I_\alpha$ such that
\begin{equation*}
\begin{split}
{\rm det}\big(\overline xE_{12}+\overline xE_{21}+{\overline x}^2E_{22}+(I_n-\beta)\big)&=\overline xa+{\rm det}\big(\overline xE_{21}+{\overline x}^2E_{22}+(I_n-\beta)\big)\\
&=\overline xa+\overline xb+{\rm det}\big({\overline x}^2E_{22}+(I_n-\beta)\big)\\
&=\overline xa+\overline xb+{\overline x}^2c+{\rm det}\big(I_n-\beta\big).
\end{split}
\end{equation*}
It follows that $\overline x(a+b+\overline xc)=-{\rm det}(I_n-\beta)$. 
As $\beta$ is nilpotent, $I_n-\beta$ is invertible, ${\rm det}(I_n-\beta)$ is a unit in $R/I_\alpha$. Hence, we deduce that $\overline x\in U(R/I_\alpha)$, a contradiction.		Thus, we have proved that $R/I_\alpha$ is a field.

Next we show that $R/I_\alpha \cong{\mathbb Z}_2$.	Let $\overline y\in R/I_\alpha$.
Then   the matrix
 \begin{equation*}
\begin{split}
V=\begin{pmatrix}0&1&0&\cdots &0\\
                       1&y&0&\cdots &0\\
                       0&0&1&\cdots &0\\
                       \vdots&\vdots&\vdots&\ddots&\vdots\\
                       0&0&0&\cdots&1\end{pmatrix}
\end{split}
\end{equation*} is a unit in $S$, so $V$ is nil-clean in $S$. Hence $\overline V$ is nil-clean in $\overline S$. Write  $\overline V=\beta+\epsilon$ where $\epsilon^2=\epsilon\in \overline S$ and $\beta\in {\rm Nil}(\overline S)$. 	By \cite{S}, $\epsilon$ is similar to a diagonal matrix, so it is similar to $\begin{pmatrix}I_k&\bf 0\\
                       \bf 0&\bf 0\end{pmatrix}$ for some $0< k<n$. Moreover, as $\beta$ is nilpotent, it is similar to a strictly upper triangular matrix. Thus, as the trace is similarity-invariant, we obtain that ${\rm trace}(\beta)=\bar 0$ and ${\rm trace}(\epsilon)=\bar k$. As $\bar 2=\bar0$ in $\overline S$, for any $m\in \mathbb Z$, $\overline m=\bar 0$ or $\overline m=\bar 1$. So we see that ${\rm trace}(\overline V)=\bar y+\overline{n-2}$, which is equal to $\bar y$ or $\bar y+\bar 1$, and that ${\rm trace}(\epsilon)$ is equal to $\bar 0$ or $\bar 1$. Therefore, from ${\rm trace}(\overline V)={\rm trace}({\beta})+{\rm trace}(\epsilon)$, we deduce that $\bar y=\bar 0$ or $\bar y=\bar 1$. Hence $R/I_\alpha\cong \mathbb Z_2$.						 
\end{proof}

The matrix ring over a UNC ring need not be a UNC ring.

\begin{exa}\label{exa3.3} For any $n\ge 2$, ${\mathbb M}_n\big(\mathbb Z_2[t]\big)$ is not a UNC ring, while ${\mathbb Z}_2[t]$ is a UNC ring. 
\end{exa}	
In \cite[Theorem 2.6]{DL}, it is proved that any unital subring of a UU ring is again a UU ring. 
But a subring of a UNC ring may not be a UNC ring.

\begin{exa}\label{exa3.4}
Take $u=\begin{pmatrix}0&0&1\\
                       1&0&0\\
                       0&1&0\end{pmatrix}\in {\mathbb M}_3(\mathbb Z_2)$. Then $u^3=1$. Let $R$ be the unital subring of ${\mathbb M}_3(\mathbb Z_2)$ generated by $u$. That is, 
\begin{equation*}
\begin{split}
R&=\{a1+bu+cu^2: a,b,c\in \mathbb Z\}\\
&=\{0, 1, u, u^2, 1+u, 1+u^2, u+u^2, 1+u+u^2\}.
\end{split}
\end{equation*}
One easily sees that $R$ is reduced. As $u^2\not= u$, $u$ is not nil-clean. So $R$ is not a UNC ring, but 
${\mathbb M}_3(\mathbb Z_2)$ is a UNC ring.   
\end{exa}

\subsection{UNC group rings}

(Strongly) nil-clean group rings have been discussed in \cite{KWZ16, MRS, STZ}. Here we consider when a group ring is a UNC or UU ring 
following the idea in \cite{STZ}. A group $G$ is called locally finite if every finitely generated subgroup of $G$ is finite. Let $p$ be a prime  number. A group $G$ is called a $p$-group if the order of each element of $G$ is a power	of $p$.  The center of a group $G$ is denoted by ${\mathcal Z}(G)$.

For the group ring $RG$ of a group $G$ over a ring $R$, the ring homomorphism $ \omega :RG\rightarrow R$, $\Sigma r_gg\mapsto \Sigma r_g$ is called the augmentation map, and the kernel ${\rm ker}(\omega)$ is called the augmentation ideal of the group ring $RG$ and is denoted by $\triangle (RG)$. Note that $\triangle (RG)$ is an ideal of $RG$ generated by the set $\{1-g: g\in G\}$.
	
\begin{prop}\label{prop4.1}
If $R$ is a UNC ring  and $G$ is a locally finite $2$-group, then $RG$ is a UNC ring.
\end{prop}
\begin{proof} As $G$ is locally finite,  to show that $RG$ is a UNC ring it suffices to show that $GH$ is a UNC ring for any finite subgroup $H$ of $G$.  So, without loss of generality, one can assume that $G$ is a finite $2$-group.  As $R$ is a UNC ring, $2\in J(R)$ is nilpotent by Lemma \ref{lem2.4}. So, by \cite[Theorem 9]{C63}, $\triangle (RG)$ is nilpotent. As $RG/\triangle (RG)\cong R$, it follows from Theorem \ref{thm2.5} that $RG$ is a UNC ring.
\end{proof}

The hypercenter of a group $G$, denoted by $H(G)$, is defined to be the union of 
the (transfinite) upper central series of the group $G$. 
\begin{thm}\label{thm4.2}
Let $R$ be a ring and $G$ be a group. If $RG$ is a UNC ring, then $R$ is a UNC ring and $H(G)$ is a $2$-group.
\end{thm}
\begin{proof} Let 
$1={\mathcal Z}_0(G)\subseteq {\mathcal Z}_1(G)\subseteq {\mathcal Z}_2(G)\subseteq \ldots\subseteq {\mathcal Z}_\alpha(G)=H(G)$
be the upper central series of length $\alpha$ for $G$. Clearly, $\mathcal Z_0(G)$ is a $2$-group. Assume that, for some $\beta\le \alpha$, ${\mathcal Z}_\sigma(G)$ is a $2$-group for all $\sigma<\beta$. We next verify that ${\mathcal Z}_\beta(G)$ is a $2$-group. 
This is certainly true if $\beta$ is a limit ordinal. If $\beta$ is not a limit ordinal, then ${\mathcal Z}_{\beta-1}(G)$ is a $2$-group, and
${\mathcal Z}_\beta(G)/{\mathcal Z}_{\beta-1}(G)={\mathcal Z}(\overline G)$, where $\overline G=G/{\mathcal Z}_{\beta-1}(G)$. Let $\overline R=R/2R$. As $\overline R\overline G$ is an image of $RG$, it is a UNC ring. So, for $g\in {\mathcal Z}(\overline G)$, $g$ is nil-clean in $\overline R\overline G$. As $g$ is central, it is strongly nil-clean. Hence, by Lemma \ref{lem2.8}, 
$g(1-g)=g-g^2$ is nilpotent. It follows that $1-g\in \overline R\overline G$ is nilpotent. So, for some $n>0$,
$(1-g)^{2^n}=0$. That is, $g^{2^n}=1$. Hence,  ${\mathcal Z}_\beta(G)/{\mathcal Z}_{\beta-1}(G)={\mathcal Z}(\overline G)$ is a $2$-group. As ${\mathcal Z}_{\beta-1}(G)$ is a $2$-group, it follows that ${\mathcal Z}_{\beta}(G)$ is a $2$-group. Therefore, by the Transfinite Induction, $H(G)={\mathcal Z}_\alpha(G)$ is a $2$-group.
\end{proof}

A nilpotent group is a group $G$ such that $G=  {\mathcal Z}_n(G)$ for a finite number $n$.
\begin{thm}\label{thm4.3}
Let $R$ be a ring and $G$ be a nilpotent group. Then
$RG$ is a UNC ring if and only if $R$ is a UNC ring and $G$ is a $2$-group. 
\end{thm}
\begin{proof} The claim follows from Proposition \ref{prop4.1} and Theorem \ref{thm4.2}.
\end{proof}

\begin{thm}\label{thm4.4} If $RG$ is a UU ring, then $R$ is a UU ring and $G$ is a $2$-group. The converse holds if $G$ is locally finite.
\end{thm}
\begin{proof}
If $RG$ is a UU ring, then as an image of $RG$, $R$ is certainly a UU ring. Moreover, $G$ is a $2$-group by the same argument 
as in the proof of Theorem \ref{thm4.2} or by \cite{DL}. Conversely,  as $R$ is a UU ring, $2\in J(R)$ is nilpotent by Lemma \ref{lem2.4} or \cite{DL}.
As $G$ is a locally finite $2$-group, $\triangle (RG)$ is locally nilpotent by \cite[Corollary, p.682]{C63}. Let $x\in U(RG)$. 
Then $\omega(x)\in U(R)$ is strongly nil-clean, so $\omega(x)-\omega(x)^2\in R$ is nilpotent.  Hence, for some $n>0$,
$\omega((x-x^2)^n)=(\omega(x-x^2))^n=0$. So
$(x-x^2)^n\in \triangle (RG)$. It follows that $x-x^2$ is nilpotent. 
So $x$ is strongly nil-clean by Lemma \ref{lem2.8}. 
\end{proof}

\subsection{Semipotent UU rings}
An element $a$ in a ring $R$ is called uniquely nil-clean if there exists a unique idempotent $e$ in $R$ such that $a-e$ is nilpotent, and the ring $R$ is called uniquely nil-clean if each element of $R$ is uniquely nil-clean. Uniquely nil-clean rings were characterized in Diesl \cite{D13}.  Here, using a recent result of \^{S}ter in \cite{Ster},   we first show that UU rings are exactly those rings whose units are uniquely nil-clean.

\begin{thm}\label{thm2.23}
A ring $R$ is a UU ring if and only if every unit of $R$ is uniquely nil-clean. 
\end{thm}
\begin{proof}
$(\Rightarrow)$. Let $u\in U(R)$. By the hypothesis, $u=1+b$ with $b\in {\rm Nil}(R)$. Assume that $u=e+x$ where $e^2=e$ and $x\in {\rm Nil}(R)$.
Then $1-e=x-b$. By \^{S}ter \cite[Corollary 2.13]{Ster}, ${\rm Nil}(R)$ is closed under addition. So $x-b\in {\rm Nil}(R)$, and hence $1-e$ is nilpotent. It follows that $e=1$ and $x=b$. Hence $u$ is uniquely nil-clean.

$(\Leftarrow)$. Let $u\in U(R)$. Then there exist $e^2=e\in R$ and $x\in {\rm Nil}(R)$ such that $u=e+x$. Thus, $u=u^{-1}eu+u^{-1}xu$ is another 
nil-clean decomposition in $R$. So it follows that $e=u^{-1}eu$.
This gives $eu=ue$. So $u$ is strongly nil-clean.  
\end{proof}

In contrast to Theorem \ref{thm2.23}, a unipotent unit need not be  uniquely nil-clean, even in a nil-clean ring.
\begin{exa}\label{exa2.24}
Let $R={\mathbb M}_3(\mathbb Z_2)$. Then $R$ is a nil-clean ring. As observed 
in \cite{FPS}, if $A$ is the strictly upper triangular matrix in $R$ whose all entries
above the diagonal are equal to $1$ and let $B$ be the transpose of $A$, then $A, B$ are nilpotent and
$E:=A+B=\begin{pmatrix}0&1&1\\
                       1&0&1\\
                       1&1&0\end{pmatrix}$ is a nonzero idempotent. Thus, $U:=I_3-A=(I_3-E)+B$ is a unipotent unit that is not uniquely nil-clean. 
\end{exa}

A ring $R$ is strongly $\pi$-regular if, for each $a\in R$, $a^n\in a^{n+1}R\cap Ra^{n+1}$ for some $n\ge 1$. A ring is clean if every element is a sum of a unit and an idempotent, and a ring $R$ is an exchange ring if, for each $a\in R$, $a-e\in (a-a^2)R$ for some $e^2=e\in R$ (see \cite{N77}).  A ring $R$ is semipotent if every right ideal 
not contained in $J(R)$ contains a non-zero idempotent. We have the implications: strongly $\pi$-regular $\Rightarrow $ clean $\Rightarrow$ exchange $\Rightarrow $ semipotent; and none of the arrows is reversible.

Diesl \cite[Corollary 3.11]{D13} proved that a ring $R$ is strongly nil-clean if and only if $R$ is a strongly $\pi$-regular, UU ring. Danchev and Lam \cite[Theorem 4.3]{DL} generalized the direction ``$\Leftarrow$'' by showing that a ring $R$ is strongly nil-clean if and only if $R$ is an exchange (or clean), UU ring. We now generalize the direction ``$\Leftarrow$'' further by showing the following

\begin{thm}\label{thm2.25} A ring $R$ is a strongly nil-clean ring if and only if it is a semipotent, UU ring.
\end{thm}
\begin{proof}
$(\Rightarrow )$. The implication is clear.

$(\Leftarrow )$.  Let $R$ be a semipotent, UU ring. So $R/J(R)$ is a semipotent, UU ring \big(by \cite[Theorem 2.4(2)]{DL}\big), and moreover, $J(R)$ is nil. So $R$ is strongly nil-clean if and only if so is $R/J(R)$ by \cite[Theorem 2.7]{KWZ16}. Hence, to show the necessity, we can assume that $J(R)=0$. 

We now show that $R$ is a reduced ring. Assume $a^2=0$ for some $0\not= a\in R$.
As $R$ is a semipotent ring with $J(R)=0$, there exists 
$e^2=e\in R$ such that $eRe$ is isomorphic to a $2\times 2$ matrix ring over a non trivial ring by Levitzki \cite[Theorem 2.1]{L53}.
But, as $eRe$ is again a UU ring by \cite[Proposition 2.5]{C}, this gives  a contradiction to \cite[Corollary 3.3]{C}. 
Hence $R$ is reduced. It follows that  $U(R)=\{1\}$. 

We next show that $R$ is a Boolean ring. Assume on the contrary that $a^2\not= a$ for some $a\in R$.
As $R$ is semipotent with $J(R)=0$, $(a-a^2)R$ contains a nonzero idempotent, say $e$. Write $e=(a-a^2)b$ with $b\in R$. Then $e=e(a-a^2)b=ea\cdot e(1-a)b=e(1-a)\cdot eab$.  As $eRe$ is reduced, $ea$ and $e(1-a)$ are units of $eRe$. 
As $U(R)=\{1\}$, we have $U(eRe)=\{e\}$. Hence $ea=e$ and $e(1-a)=e$.  It follows that $e=0$, a contradiction. 
\end{proof}

In \cite[Theorem 2.9]{H}, Henriksen proved that a von Neumann regular ring $R$ with $U(R)=\{1\}$ is Boolean. 
Danchev and Lam \cite[Corollary 4.2]{DL} proved that a ring $R$ is an exchange ring with $U(R)=\{1\}$ if and only if $R$ is Boolean. 

\begin{cor}\label{cor2.26} A ring  $R$ is semipotent with $U(R)=\{1\}$ if and only if $R$ is Boolean.
\end{cor}
The question of Danchev and Lam whether a clean, UNC ring is nil-clean is still open. In view of Theorem \ref{thm2.25}, \cite[Corollary 3.11]{D13} and \cite[Question 4]{D13}, the following questions arise.
\begin{ques}\label{ques2.27}
Is a semipotent, UNC ring a nil-clean ring? 
What can be said about 
strongly $\pi$-regular UNC rings?
\end{ques}

\section{Units being sums of a nilpotent and two idempotents} 

As a generalization of a strongly nil-clean ring, a strongly $2$-nil-clean ring was introduced by Chen and Sheibani \cite{CS} to be the ring for which every element is a sum of a nilpotent and two idempotents that commute with one another. The structure of these rings is obtained in \cite{CS}. 
In this section, we consider the ``unit'' version of strongly $2$-nil-clean rings. That is, the rings for which every unit is a sum of a nilpotent and two idempotents that commute with one another. These rings extend UU rings, and will be completely characterized here. A special situation of a strongly $2$-nil-clean ring is the property
that every element of a ring is a sum of two commuting idempotents, first considered by Hirano and Tominaga \cite{HT}. Here the ``unit'' version of this property is also discussed.

\subsection{Units being sums of a nilpotent and two idempotents}

\noindent

\begin{defn}\label{defn3.1}
A ring $R$ is called a UNII-ring if every unit of $R$ is a sum of a nilpotent and two idempotents.
The ring $R$ is called a strong UNII-ring if for each $u\in U(R)$, $u=b+e+f$ where $b\in {\rm Nil}(R)$
, $e,f\in {\rm idem}(R)$ such that $b,e,f$ commute.
\end{defn}

\begin{lem}\label{lem3.2}
A ring $R$ is strong UNII-ring if and only if $R=A\oplus B$ where $A,B$ are strong UNII-rings, $2\in J(A)$ is nilpotent and $3\in J(B)$ is nilpotent. 
\end{lem}
\begin{proof}
$(\Leftarrow)$. The implication is clear.

$(\Rightarrow)$. Write $-1=b+e+f$ where $b\in {\rm Nil}(R)$, $e,f\in {\rm idem}(R)$ and $b,e,f$ all commute.
Then 
\begin{equation*}
\begin{split}
1+2b+b^2&=(-1-b)^2=(e+f)^2=e+f+2ef=(-1-b)+2e(-1-b-e)\\
&=-1-b-4e-2eb,
\end{split}
\end{equation*}
which gives $2+4e=-3b-b^2-2eb$. So $6e=(2+4e)e=(-3b-b^2-2eb)e=-(5e+be)b$ is nilpotent.
Similarly, $6f$ is nilpotent. So there exists $n\ge 1$ such that $6^ne=0$ and $6^nf=0$.
Thus $0=6^n(e+f)=6^n(-1-b)$. As $-1-b\in U(R)$, $6^n=0$, i.e., $2^nR\cap 3^nR=0$. So $R=A\oplus B$ where $A\cong R/2^nR$ and $B\cong R/3^nR$, and $A,B$ are strong UNII-rings. 
\end{proof}

\begin{lem}\label{lem3.3} The following are equivalent for a ring $R$:
\begin{enumerate}
\item A ring $R$ is a strong UNII-ring with $2$ nilpotent.
\item Each unit of  $R$ is a sum of a nilpotent and a tripotent that commute and $2$ is nilpotent.
\item $R$ is a UU ring.
\end{enumerate}
\end{lem}
\begin{proof}
$(3)\Rightarrow (1)$. The implication is clear.

$(1)\Rightarrow (2)$. Let $u\in U(R)$ and write $u=b+e+f$ where $b\in {\rm Nil}(R)$, $e,f\in {\rm idem}(R)$ and $b,e,f$ all commute. Then $g:=e+f-2ef$ is an idempotent and $c:=u-g=b+2ef$ is nilpotent. Moreover, $g, c$ commute. So $u=c+g$ is strongly nil-clean. Hence $(2)$ holds. 

$(2)\Rightarrow (3)$.  Let $u\in U(R)$ and write $u=b+t$ where $b\in {\rm Nil}(R)$, $t^3=t$ and $bt=tb$.
Then $(t-t^2)^2=2(t^2-t)\in {\rm Nil}(R)$, so $b+(t-t^2)$ is nilpotent. Thus,
$u=(b+t-t^2)+t^2$ is strongly clean. So $R$ is a UU ring.
\end{proof}

\begin{lem}\label{lem3.4} Let $R$ be a ring with $3\in {\rm Nil}(R)$. 
The following are equivalent:
\begin{enumerate}
\item $R$ is a strong UNII-ring. 
\item Each unit of  $R$ is a sum of a nilpotent and a tripotent that commute.
\item $1-u^2$ is nilpotent for every $u\in U(R)$. 
\end{enumerate}
\end{lem}
\begin{proof}
$(1)\Rightarrow (2)$. Let $u\in U(R)$ and write $u=b+e+f$ where $b\in {\rm Nil}(R)$, $e,f\in {\rm idem}(R)$ and $b,e,f$ all commute. Then $g:=e+f-3ef$ is a tripotent and $b+3ef$ is nilpotent. Moreover, $g, b+3ef$ commute. So $u=(b+3ef)+g$ is a sum of a nilpotent and a tripotent that commute. Hence $(2)$ holds. 

$(2)\Rightarrow (3)$. Let $u\in U(R)$ and write $u=b+f$ where $b\in {\rm Nil}(R)$, $f$ is a tripotent and $b,f$ all commute. Then 
$u^3\equiv b^3+f\,({\rm mod}\,3R)$, so $u-u^3\equiv b-b^3\,({\rm mod}\,3R)$. As $3, b\in {\rm Nil}(R)$, 
$u-u^3\in {\rm Nil}(R)$, so $1-u^2$ is nilpotent. 

$(3)\Rightarrow (1)$. 
Let $u\in U(R)$ and write $u=b+x$ where $b\in {\rm Nil}(R)$, $x^3=x$ and $xb=bx$.  Then $u=b+x^2+(x-x^2)$, and 
$(x-x^2)-(x-x^2)^2=3(x-x^2)\in J(R)$ is nilpotent.
By \cite[Lemma 3.5]{YKZ16}, there exists a polynomial $\theta(t)\in {\mathbb Z}[t]$ such that $\theta(x)^2=\theta(x)$ and $j:=(x-x^2)-\theta(x)\in J(R)$. Thus, $u=(b+j)+x^2+\theta(x)$, where $x^2, \theta(x)$ are idempotents, and 
$b+j$ is nilpotent and they commute with each other.
\end{proof}

\begin{thm}\label{thm3.5}
The following are equivalent for a ring $R$:
\begin{enumerate}
\item $R$ is a strong UNII-ring.
\item $1-u^2$ is nilpotent for every $u\in U(R)$ and $6$ is nilpotent.
\item $R$ is one of the following types:
\begin{enumerate}
\item $R$ is a UU ring.
\item $1-u^2$ is nilpotent for every $u\in U(R)$ and $3$ is nilpotent.
\item $R=A\oplus B$, where $A$ is a UU ring, $1-u^2$ is nilpotent for every $u\in U(B)$ and $3\in B$ is nilpotent.
\end{enumerate}	
\end{enumerate}
\end{thm}
\begin{proof}
$(1)\Leftrightarrow (3)$. The equivalence follows from Lemmas \ref{lem3.2}-\ref{lem3.4}. 

$(3)\Rightarrow (2)$. The implication is clear.

$(2)\Rightarrow (3)$. Let $6^n=0$ with $n\ge 1$. Then $R=A\oplus B$ where $A\cong R/2^nR$ and $B\cong R/3^nR$. 
The hypothesis on $R$ shows that $1-u^2$ is nilpotent if $u\in U(A)$ or $u\in U(B)$. 
For $u\in U(A)$, $(1-u)^2=(1-u^2)+2(u^2-u)$ is nilpotent, so $1-u$ is nilpotent. This shows that $A$ is a UU ring. 
\end{proof}

The condition that $6$ is nilpotent can not be removed in Theorem \ref{thm3.5}(2):  As $U(\mathbb Z)=\{-1,1\}$, 
$1-u^2=0$ for all $u\in U(\mathbb Z)$, but $6\in \mathbb Z$ is not nilpotent. Theorem \ref{thm3.6} below can be viewed as an extension of Theorem \ref{thm2.25}.

\begin{thm}\label{thm3.6}
A ring is a semipotent, strong UNII-ring if and only if $R=A\oplus B$, where $A$ is zero or $A/J(A)$ is Boolean with $J(A)$ nil and $B$ is zero or $B/J(B)$ is a subdirect product of $\mathbb Z_3$'s with $J(B)$ nil. 
\end{thm}
\begin{proof}
$(\Leftarrow)$. As $A$ is strongly nil-clean, it is a strong UNII-ring. For any $u\in U(B)$, $u^3-u$ is nilpotent. So $(u^2)^2-u^2=u(u^3-u)$ is nilpotent and $(u-u^2)^2-(u-u^2)=(u-2)(u^3-u)+3(u^2-u)$ is nilpotent. By \cite[lemma 3.5]{YKZ16}, 
there exist $\theta_1(t), \theta_2(t)\in \mathbb Z[t]$ such that
$\theta_1(u)^2=\theta_1(u)$, $\theta_2(u)^2=\theta_2(u)$, and $u^2-\theta_1(u)$, $(u-u^2)-\theta_2(u)$ are nilpotent. So $u=\theta(u)+\theta_1(u)+\theta_2(u)$, where
$\theta(u)=(u^2-\theta_1(u))+[(u-u^2)-\theta_2(u)]$ is nilpotent. So $B$ is a strong UNII-ring. Hence $R$ is a strong UNII-ring.

$(\Rightarrow )$. By Theorem \ref{thm3.5}, $R=A\oplus B$, where $A$ is zero or $A$ is a UU ring,  and $B$ is zero or $1-u^2\in {\rm Nil}(B)$ for every $u\in U(B)$ with $3\in {\rm Nil}(B)$. We can assume that $A\not= 0$ and $B\not= 0$. As  $R$ is semipotent, $A, B$ are semipotent. So $A$ is strongly nil-clean by Theorem \ref{thm2.25}, and hence $A/J(A)$ is Boolean with $J(A)$ nil.  For $j\in J(B)$, $(1+j)^2-1=j(2+j)\in {\rm Nil}(B)$. As $2+j\in U(B)$, $j\in {\rm Nil}(B)$. So $J(B)$ is nil.  We  next show that $\overline B:=B/J(B)$ is reduced. Assume $x^2=0$ for some $0\not= x\in \overline B$. As $\overline B$ is semiprimitive semipotent, there exists 
$f^2=f\in \overline B$ such that $f\overline Bf\cong {\mathbb M}_2(S)$ for a non trivial ring $S$ by Levitzki \cite[Theorem 2.1]{L53}. 
In ${\mathbb M}_2(S)$, $\begin{pmatrix}1&1\\
                       1&0\end{pmatrix}^2-\begin{pmatrix}1&0\\
                       0&1\end{pmatrix}=\begin{pmatrix}1&1\\
                       1&0\end{pmatrix}$ is not nilpotent. So, there exists a unit $v$ of $f\overline Bf$ such that $f-v^2$ is not nilpotent. Thus $u:=v+1-f$ is a unit of $\overline B$ and $1-u^2=f-v^2$ is not nilpotent. This is a contradiction. Hence $\overline B$ is reduced. It follows that  $w^2=1$ for any $w\in U(\overline B)$. 
Next we show that $y^3=y$ in $\overline B$ for every $y\in \overline B$. Assume that
$y-y^3\not= 0$ for some $y\in \overline B$. As $\overline B$ is semipotent, there exists $0\not= e^2=e\in (y-y^3)\overline B$. Write $e=(y-y^3)z$ with $z\in \overline B$. Then $e=ey\cdot e(1-y^2)z=e(1-y)\cdot ey(1+y)z=e(1+y)\cdot ey(1-y)z$. As $e\overline Be$ is reduced, $ey, e(1-y)$ and $e(1+y)$ all are units of $e\overline Be$.
As $u^2=1$ for each $u\in \overline B$, $v^2=e$ for all $v\in e\overline Be$. Hence,
$e=(ey)^2=(e(1+y))^2=(e(1-y))^2$.  It follows that $2e=0$. As $2$ is a unit of $e\overline Be$, $e=0$. This contradiction shows that  $y^3=y$ in $\overline B$ for every $y\in \overline B$. As $3=0$ in $\overline B$, $\overline B$ is a subdirect product of $\mathbb Z_3$'s (see \cite[Ex.12.11]{L}).					 
\end{proof}

\begin{cor}
A ring is a strongly $2$-nil-clean ring if and only if it is a semipotent, strong UNII-ring. 
\end{cor}

\begin{proof}
This is by Theorem \ref{thm3.6} and \cite[Theorem 4.2]{CS}.
\end{proof}

\begin{cor} If $R$ is a semipotent strong UNII-ring, then every element of $R$ is a sum of a nilpotent and two idempotents that commute with one another.
\end{cor}
\begin{proof}
This is by Theorem \ref{thm3.6} and \cite[Theorem 4.5]{CS}.
\end{proof}

\subsection{Units being sums of two idempotents}

\noindent

Hirano and Tominaga \cite{HT} have characterized the rings in which every element is the sum of two commuting idempotents.
\begin{defn}\label{defn3.7}
A ring $R$ is called a UII-ring if every unit of $R$ is a sum of two idempotents, and
$R$ is called a strong UII-ring if every unit of $R$ is a sum of two commuting idempotents.	
\end{defn}

\begin{lem}\label{lem3.8}
A finite direct sum $R_1\oplus R_2\oplus \cdots\oplus R_n$ of rings is a {\rm (}strong{\rm )} UII-ring if and only if each $R_i$ is a {\rm (}strong{\rm )} UII-ring. 
\end{lem}

\begin{lem}\label{lem3.9}
A ring $R$ is a UII-ring if and only if $R=A\oplus B$ where $A,B$ are UII-rings, $2=0$ in $A$ and $3=0$ in $B$. 
\end{lem}
\begin{proof}
$(\Leftarrow)$. The implication is clear. 	

$(\Rightarrow)$. Write $-1=e+f$ where $e,f$ are idempotents of $R$. Then $-1-e=f=f^2=(-1-e)^2=1+3e$, so $4e+2=0$. Thus, $0=(4e+2)e=6e$. Similarly, $6f=0$. 
Hence $6=-6(e+f)=-6e-6f=0$, so $2R\cap 3R=0$. It follows that $R=A\oplus B$ where $A\cong R/2R$ and $B\cong R/3R$. Moreover, $A,B$ are clearly UII-rings.
\end{proof}	

\begin{lem}\label{lem3.10}
	A ring $R$ is strong UII with $ch(R)=2$ if and only if $U(R)=\{1\}$.
\end{lem}
\begin{proof}
	If $U(R)=\{1\}$, then $R$ is clearly strong UII, and $-1=1$; so $2=0$ in $R$. 
	If $R$ is strong UII with $2R=0$, then, for any $u\in U(R)$, $u=e+f$ where $e,f$ are commuting idempotents of $R$. Thus $u+e=f=f^2=(u+e)^2=u^2+2ue+e=u^2+e$. It follows that $u^2=u$, i.e., $u=1$.
	\end{proof}
	
	\begin{lem}\label{lem3.11} 
	A ring $R$ is a strong UII-ring with $ch(R)=3$ if and only if $U(R)=1+idem(R)$.
		\end{lem}
\begin{proof}
$(\Leftarrow)$. Suppose $U(R)=1+idem(R)$. Then $2=1+1\in U(R)$. As $-1=1+(-2)$, we infer that $(-2)^2=-2$, so $-2=1$. That is, $3=0$ in $R$. Moreover, $R$ is clearly a strong UII-ring.

	$(\Rightarrow )$. For $e^2=e\in R$, $1+e=1-2e\in U(R)$. Let $u\in U(R)$. Write $u=e+f$ where $e,f$ are commuting idempotents. Then $u-e=f=f^2=(u-e)^2=u^2-2ue+e$, so 
	$$e=eu-u+u^2.$$
	Multiplying the equality by $e$ from the left gives  $e=eu^2$. It follows that $eu^2=eu-u+u^2$, and hence 
	$eu=e-1+u$, i.e., $e=1-u+eu$.  We deduce that $1-u+eu=eu-u+u^2$, so $u^2=1$. This gives that
$1=u^2=(e+f)^2=e+f+2ef=u+2ef$. That is, $u=1+g$ with $g=ef$.
	\end{proof}
	
\begin{thm}\label{thm3.12}
A ring $R$ is a strong UII-ring if and only if $R$ is one of the following types:
\begin{enumerate}
\item $U(R)=\{1\}$.
\item $U(R)=1+{\rm idem}(R)$.
\item $R=A\oplus B$ where $U(A)=\{1\}$ and $U(B)=1+{\rm idem}(B)$.
\end{enumerate}	
\end{thm}
\begin{proof}
It follows from Lemmas \ref{lem3.8}-\ref{lem3.11}.
\end{proof}

\begin{cor}\label{cor3.13}
If $R$ is a strong UII-ring, then $J(R)=0$ and ${\rm Nil}(R)=0$. 
\end{cor}
\begin{proof}
One easily sees that $U(R)=\{1\}$ or $U(R)=1+{\rm idem}(R)$ implies that $J(R)=0$ and ${\rm Nil}(R)=0$.
\end{proof}

In view of Corollary \ref{cor3.13}, we see that, for any ring $R$ and $n\ge 2$,  the rings ${\mathbb M}_n(R)$,  ${\mathbb T}_n(R)$ and $R[[t]]$ are not strong UII-rings. 

\begin{prop}\label{prop3.14}
Let  $\sigma$ be an endomorphism of a ring $R$.  Then $R[t; \sigma]$ is a strong UII-ring if and only if $R$ is a strong UII-ring and $\sigma(e)=e$ for all $e\in {\rm idem}(R)$. 
\end{prop}
\begin{proof}
$(\Rightarrow )$. For $u\in U(R)$, $u\in U(R[t;\sigma])$, so $u=\sum_{i\ge 0}a_it^i+\sum_{i\ge 0}b_it^i$ is a sum of two
commuting idempotents in $R[t;\sigma]$. Thus $u=a_0+b_0$ is a sum of two commuting idempotents in $R$. So $R$ is a strong UII-ring. Since $R[t;\sigma]$ is a reduced ring, each $e\in {\rm idem}(R)$ is central in $R[t;\sigma]$. So 
$et=te=\sigma(e)t$, showing that $\sigma(e)=e$.

$(\Leftarrow)$. As $R$ is a strong UII-ring, $R$ is a reduced ring. So idempotents of $R$ are central, and hence units of $R$ are central.
Moreover, the assumption that $\sigma(e)=e$ for all $e^2=e\in R$ implies that $\sigma(u)=u$ for all $u\in U(R)$. Now one can easily show that  ${\rm idem}(R[t;\sigma])={\rm idem}(R)$ and $U(R[t;\sigma])=U(R)$. It follows that $R[t; \sigma]$ is a strong UII-ring.
\end{proof}

\begin{prop}\label{prop3.15}
Let $R$ be a ring and $G$ a nontrivial group. Then $RG$ is a strong UII-ring if and only if $U(R)=1+{\rm idem}(R)$ and $G$ is a group of exponent $2$. 
\end{prop}

\begin{proof}
$(\Rightarrow )$. Let $u\in U(R)$. Then $u=\alpha+\beta$, where $\alpha, \beta$ are commuting idempotents in $RG$. So
$u=\omega(\alpha)+\omega(\beta)$ is a sum of two commuting idempotents in $R$. So $R$ is a strong UII-ring. By Theorem \ref{thm3.12}, $R=A\oplus B$ where
$A,B$ are strong UII-rings, $2=0$ in $A$ and $3=0$ in $B$. Thus $RG\cong AG\oplus BG$, so $AG$ is a strong UII-ring. As $2=0$ in $AG$, 
$U(AG)=\{1\}$ by  Lemma \ref{lem3.10}. Because $G$ is a nontrivial group, it must be that $A=0$. So $R=B$, and hence $U(R)=1+{\rm idem}(R)$. 
As $3=0$ in $RG$, $U(RG)=1+{\rm idem}(RG)$ by Lemma \ref{lem3.11}.  It follows that $g^2=1$ for all $g\in G$.

$(\Leftarrow)$. Since $G$ is a group of exponent $2$, $G$ is locally finite. So it suffices to show that for any finite subgroup $H$ of $G$, $RH$ is a strong $UII$-ring. But $2\in U(R)$ by Lemma \ref{lem3.11}, so $RH$ is a direct sum of finitely many copies of $R$. Hence $RH$ is a strong UII-ring.
\end{proof}

By  Hirano and Tominaga \cite{HT}, the ring $R$ must be commutative if every element of $R$ is a sum of two commuting idempotents. However, a strong UII-ring need not be commutative. In fact, by Henriksen \cite[pp.86]{H}, the Weyl algebra over $\mathbb Z_2$ in the noncommuting indeterminants $x,y$ subject to the relation $yx=xy+1$ is a strong UII-ring that is not commutative.

\section*{Acknowledgments} The authors thank the referee for valuable comments and suggestions. 
Karimi-Mansoub acknowledges Tarbiat
Modares University (Iran) to grant her a Ph.D scholarship for visiting Memorial University of Newfoundland.
Zhou's research was supported by a Discovery Grant from NSERC of Canada.

\end{document}